\documentclass[10pt,reqno]{amsart}
\usepackage{indentfirst, amssymb,amsmath,amsthm}
\newtheorem*{theoA}{Theorem A}
\newtheorem*{theoB}{Theorem B}
\newtheorem*{theoC}{Theorem C}
\newtheorem*{theoD}{Theorem D}
\newtheorem*{theoE}{Theorem E}

\newtheorem{theo}{Theorem}
\newtheorem{lem}{Lemma}

\newtheorem{ex}{Example}
\newtheorem{rem}{Remark}

\newcommand{\ol}{\overline}
\newcommand{\be}{\begin{equation}}
\newcommand{\ee}{\end{equation}}
\newcommand{\beas}{\begin{eqnarray*}}
\newcommand{\eeas}{\end{eqnarray*}}
\newcommand{\bea}{\begin{eqnarray}}
\newcommand{\eea}{\end{eqnarray}}
\renewcommand{\epsilon}{\varepsilon}
\numberwithin{equation}{section}
\numberwithin{lem}{section}
\numberwithin{theo}{section}
\numberwithin{cor}{section}
\numberwithin{ex}{section}
\numberwithin{defi}{section}
\numberwithin{rem}{section}
\numberwithin{note}{section}

\begin{document}
\title[Characterization of exponential polynomial as solution...]{Characterization of exponential polynomial as solution of certain type of non-linear delay-differential equation}
\author[A. Banerjee and T. Biswas]{Abhijit Banerjee and Tania Biswas}
\date{}
\address{Department of Mathematics, University of Kalyani, West Bengal 741235, India.}
\email{abanerjee\_kal@yahoo.co.in, taniabiswas2394@gmail.com}
\renewcommand{\thefootnote}{}
\renewcommand{\thefootnote}{\arabic{footnote}}
\setcounter{footnote}{0}
\begin{abstract} 
	In this paper, we have characterized the nature and form of solutions of the following non-linear delay-differential equation: $$f^{n}(z)+\sum_{i=1}^{n-1}b_{i}f^{i}(z)+q(z)e^{Q(z)}L(z,f)=P(z),$$ where $b_i\in\mathbb{C}$, $L(z,f)$ be a linear delay-differential polynomial of $f$; $n$ be positive integers; $q$, $Q$ and $P$ respectively be non-zero, non-constant and any polynomials. Different special cases of our result will accommodate all the results of ([J. Math. Anal. Appl., 452(2017), 1128-1144.], [Mediterr. J. Math., 13(2016), 3015-3027], [Open Math., 18(2020), 1292-1301]). Thus our result  can be considered as an improvement of all of them. We have also illustrated a handful number of examples to show that all the cases as demonstrated in our theorem actually occurs and consequently the same are automatically applicable to the previous results.
\end{abstract}
\thanks{Typeset by \AmS -\LaTeX}
\maketitle
\section{Introduction, Results and Examples} 
Throughout the paper, we denote by $f$ a meromorphic function in the complex plane $\mathbb{C}$ and related to the function, we assume that the readers are familiar with the basic terms like $T(r,f)$, $N(r,f)$, $m(r, f)$, of Nevanlinna value distribution theory of meromorphic functions (see \cite{Hayman_Oxford,Laine_Gruyter}). The notation $S(r,f)$, will be used to define any quantity satisfying $S(r,f)=o(T(r,f))$ as $r\rightarrow \infty$, possibly outside a set $E$ of $r$ of finite logarithmic measure. In addition, we will respectively use the symbols $\rho(f)$, $\lambda(f)$ and $\tau(f)$ to denote the order, exponent of convergent and type of $f$. The symbol $L(f)$ will be used to represent a linear differential polynomial in $f$ with polynomial coefficients.  Also, throughout this paper, by $card(S)$, we mean the cardinality of a set $S$, i.e., the number of elements in $S$.\par

Considering the non-linear differential equation $$L(f)-p(z)f^n(z)=h(z),$$ in 2001, Yang \cite{Yang-Aust} investigated about the transcendental finite order entire solutions $f$ of the  equation, where $p(z)$ is a non-vanishing polynomial, $h(z)$ is entire and $n\geq 4$ is an integer.\par 
In 2010, Yang-Laine \cite{Yang-Laine-Jpn} 
showed that the equation $$f(z)^2 + q(z)f(z + 1) = p(z),$$ where $p(z)$, $q(z)$ are polynomials, admits no transcendental entire solutions of finite order.\par In the last two decades researchers mainly studied (see \cite{Biswas-Banerjee,Li-Lu-Xu,Liao-Yang-Zhang,Yang-Aust,Zhang-Liao} etc.) about the following three distinct features of solutions of shift or delay-differential or differential equations:\\
i) existence and non-existence conditions,\\ ii) order of growth and \\ iii) different types of forms of solutions. \par
Next, let us consider the exponential polynomial $f(z)$, defined by the form \bea\label{e1.1} f(z) = P_1(z)e^{Q_1(z)} + \cdots + P_k(z)e^{Q_k(z)},\eea
where $P_j$'s and $Q_j$'s are polynomials in $z$. Steinmetz \cite{Steinmetz}, showed that (\ref{e1.1}) can be written in the normalized form
\bea\label{e1.2}f(z) = H_0(z) + H_1(z)e^{\omega_1z^t}+\cdots + H_m(z)e^{\omega_mz^t},\eea where $H_j$ are either exponential polynomials of order $< t$ or ordinary polynomials in $z$, the leading coefficients $\omega_j$ are pairwise distinct and $m\leq k$.\par
Let $co(\mathcal{W})$ be the convex hull of a set $\mathcal{W} \subset\mathbb{C}$ which is the intersection of all convex sets containing $\mathcal{W}$. If $\mathcal{W}$ contains finitely many elements then $co(\mathcal{W})$ is obtained as an intersection of finitely many half-planes, then $co(\mathcal{W})$ is either a compact polygon with a non-empty interior or a line segment. We denote by $C(co(\mathcal{W}))$, the circumference of $co(\mathcal{W})$. If $co(\mathcal{W})$ is a line-segment, then $C(co(\mathcal{W}))$ is equals to twice the length of this line segment. Throughout the paper, we denote $W =\{\bar{\omega}_1,\bar{\omega}_2,\ldots,\bar{\omega}_m\}$ and $W_0 = W \cup \{0\}$.\par 
Now-a-days, to find the form of exponential polynomials as solution of certain non-linear differential-difference equation has become an interesting topic among researchers  (see \cite{Chen-Gao-Zhang_CMFT1,Chen-Hu-Wang_CMFT2,Xu-Rong}). Most probably, in this regard, the first attempt was made by Wen-Heittokangas-Laine \cite{Wen-Heittokangas-Laine}. In 2012, they considered the equation \bea\label{e1.3}f(z)^n + q(z)e^{Q(z)}f(z + c) = P(z),\eea
where $q(z)$, $Q(z)$, $P(z)$ are polynomials, $n\geq 2$ is an integer, $c\in\mathbb{C}\backslash \{0\}$. Wen-Heittokangas-Laine also pointed out that for a non-constant polynomial $\alpha(z)$ and $d  \in\mathbb{C}$, every solution $f$ of the form (\ref{e1.2}) reduces to a function which belongs to one of the following classes:
\beas\Gamma_1 &=&\{e^{\alpha(z)} + d \},\\
\Gamma_0 &=& \{e^{\alpha(z)}\}\eeas and classified finite order entire(meromorphic) solutions of (\ref{e1.3}) as follows: 
\begin{theoA}\cite{Wen-Heittokangas-Laine}
	Let $n \geq 2$ be an integer, let $c \in\mathbb{C}\backslash \{0\}$, $q(z)$, $Q(z)$, $P(z)$ be polynomials such that $Q(z)$ is not a constant and $q(z) \not\equiv 0$. Then the finite order entire solutions $f$ of equation (\ref{e1.3}) satisfies the following conclusions:
	\begin{itemize}
		\item [(a)] Every solution $f$ satisfies $\rho(f) = \deg Q$ and is of mean type.
		\item [(b)] Every solution $f$ satisfies $\lambda(f) = \rho(f)$ if and only if $P(z) \not\equiv 0$.
		\item [(c)] A solution $f$ belongs to $\Gamma_0$ if and only if $P(z) \equiv 0$. In particular, this is the case if $n \geq 3$.
		\item [(d)] If a solution $f$ belongs to $\Gamma_0$ and if $g$ is any other finite order entire solution of (\ref{e1.3}), then $f = \eta g$, where $\eta^{ n-1} = 1$.
		\item [(e)] If $f$ is an exponential polynomial solution of the form (\ref{e1.2}), then $f \in \Gamma_1$. Moreover, if $f\in \Gamma_1\backslash \Gamma_0$, then $\rho(f) = 1$.
	\end{itemize}
\end{theoA}
Inspired by {\it Theorem A}, in 2016, Liu \cite{Liu_Mediterr_2016} replaced $f(z+c)$ by $f^{(k)}(z+c)$ in (\ref{e1.3}) and for two polynomials $p_1(z)$, $p_2(z)$ and a non-constant polynomial $\alpha(z)$, introduced two new classes of solutions:
\beas\Gamma_1' &=&\{p_1(z)e^{\alpha(z)} + p_2(z) \},\\
\Gamma_0' &=& \{p_1(z)e^{\alpha(z)}\}\eeas
to obtain the following theorem.
\begin{theoB}\cite{Liu_Mediterr_2016}
	Under the same situation as in {\em Theorem A} with $k\geq 1$, the finite-order transcendental	entire solution $f$ of \bea\label{e1.4}f(z)^n + q(z)e^{Q(z)}f^{(k)}(z + c) = P(z),\eea should satisfy the results {\em (a), (b), (d)} and
	\begin{itemize}
		\item [(1)] A solution $f$ belongs to $\Gamma_0'$ if and only if $P(z) \equiv 0$. In particular, this is the case if $n \geq 3$.
		\item [(2)] If $f$ is an exponential polynomial solution of (\ref{e1.4}) of the form (\ref{e1.2}), then $f \in \Gamma_1'$.
	\end{itemize}
\end{theoB}
Recently, Liu-Mao-Zheng \cite{Liu-Mao-Zheng_OpenMath} considered $\Delta_cf(z)$ instead of $f(z+c)$ in (\ref{e1.3}) and proved that
\begin{theoC}\cite{Liu-Mao-Zheng_OpenMath}
	Under the same situation as in {\em Theorem A}, the finite order entire solutions $f$ of the equation \bea\label{e1.5}f(z)^n + q(z)e^{Q(z)}\Delta_cf(z) = P(z),\eea satisfies the results {\em (a), (b)} and
	\begin{itemize}
		\item [(1)] $\lambda(f) = \rho(f)-1$ if and only if $P(z) \equiv 0$. In particular, this is the case if $n \geq 3$.
		\item [(2)] If $n\geq3$ or $P(z) \equiv 0$, $f$ is of the form $f(z)=A(z)e^{\omega z^s}$, where $s=\deg{Q}$, $\omega$ is a nonzero constant and $A(z) (\not\equiv 0)$ is an entire function satisfying $\lambda(A) = \rho(A)=\deg{Q}-1$. In particular, if $\deg{Q}= 1$, then $A(z)$ reduces to a polynomial.
		\item [(3)] If $f$ is an exponential polynomial solution of (\ref{e1.5}) of the form (\ref{e1.2}), then f is of the form
		$$f(z) = H_0(z) + H_1(z)e^{\omega_1z},$$ where $H_1(z)$, $H_2(z)$ are non-constant polynomials and $\omega_1$ is a non-zero constant satisfying $e^{\omega_1c}=1$.
	\end{itemize}
\end{theoC}
In 2017, Li-Yang \cite{Li-Yang_jmaa_2017} considered the following form of equation \bea\label{e1.6}f^n(z)+a_{n-1}f^{n-1}(z)+\cdots+a_1f(z)+q(z)e^{Q(z)}f(z+c)=P(z),\eea where $a_i\in\mathbb{C}$ and proved the following results.
\begin{theoD}\cite{Li-Yang_jmaa_2017}
	Under the same situation as in {\em Theorem A}, the finite order entire solutions $f$ of equation (\ref{e1.6}) satisfies the results {\em (a), (d)} and
	\begin{itemize}
		\item [(1)] If zero is a Borel exceptional value of $f(z)$, then we have $a_{n-1}=\cdots=a_1=P(z)\equiv 0$.
		\item [(2)] If $P(z)\equiv 0$, then we have $z^{n-1}+a_{n-1}z^{n-2}+\cdots+a_1=(z+a_{n-1}/n)^{n-1}$. Furthermore, if there exists $i_0\in\{1,\ldots, n-1\}$ such that $a_{i_0}=0$, then all of the $a_j (j=1,\ldots, n-1)$ must be zero as well and we have $\lambda(f) <\rho(f)$; otherwise we have $\lambda(f) =\rho(f)$.
		\item [(3)] A solution $f$ belongs to $\Gamma_0$ if and only if $P(z) \equiv 0$ and there exists an $i_0\in\{1, \ldots, n-1\}$ such that $a_{i_0}=0$.
		\item [(4)] When $n \geq 3$, if there exists an $i_0\in\{1, \ldots, n-1\}$ such that $a_{i_0}=0$ and $ card \{z: p(z) =p'(z) =p''(z) =0\} \geq 1$ or $ card \{z: p(z) =p'(z) =0\} \geq 2$, where $p(z) =z^n+a_{n-1}z^{n-1}+\cdots+a_1z$, then $f$ belongs to $\Gamma_0$ and $a_{n-1}=\cdots=a_1= 0\equiv P(z)$.
	\end{itemize}
\end{theoD}
In the same paper, Li-Yang \cite{Li-Yang_jmaa_2017} also proved the following result.
\begin{theoE}\cite{Li-Yang_jmaa_2017}
	If $f$ is an exponential polynomial solution of the form (\ref{e1.2}) of the equation (\ref{e1.6}) for $n=2$ and $a_1\neq 0$, then the following conclusions hold.
	\begin{itemize}
		\item [(1)] when $m\geq 2$, there exists $i,j\in\{1,2,\ldots,m\}$ such that $\omega_i=2\omega_j$.
		\item [(2)] when $m =1$, then $f\in\Gamma_1$. Moreover, if $f\in\Gamma_1\backslash\Gamma_0$, then $\rho(f) =1$, $f(z)=Ke^{\frac{1}{c}(2k\pi i-\log\frac{2d+a_1}{d}}$, $Q(z)=\frac{1}{c}(2k\pi i-\log\frac{2d+a_1}{d})z$, $q(z)=-\frac{2d+a_1}{d}$ and $d^2+a_1d =P(z)$, where $K, d \in\mathbb{C}\backslash\{0\}$ and $k\in\mathbb{Z}$.
	\end{itemize}
\end{theoE}
We now introduce the generalized linear delay-differential operator of $f(z)$, \bea\label{e1.7}L(z,f)=\sum_{i=0}^{k}b_if^{(r_i)}(z+c_i)\;(\not\equiv 0),\eea where $b_i,c_i\in\mathbb{C}$, $r_i$ are non-negative integers, $c_0=0$,  $r_0=0$. In view of the above theorems it is quiet natural to characterize the nature of exponential polynomial as solution of certain non-linear complex equation involving generalized linear delay-differential operator. In this regard, we consider the following non-linear delay-differential equation \bea\label{e1.8} f^{n}(z)+\sum_{i=1}^{n-1}a_{i}f^{i}(z)+q(z)e^{Q(z)}L(z,f)=P(z),\eea where $a_i\in\mathbb{C}$, $n$ be non-negative integers; $q$, $Q$, $P$ respectively be non-zero, non-constant, any polynomials. We also introduce, for any polynomials $p_i(z)$ and non-constant polynomials $\alpha_i(z)$, a new class of solution as follows: \beas \Gamma_2' &=&\{p_1(z)e^{\alpha_1(z)}+p_2(z)e^{\alpha_2(z)}+p_3(z) \}\eeas Now we are at a state to present our main result which improves all the above mentioned results as follows:
\begin{theo}\label{t1.1}
	Under the same situation as in {\em Theorem A}, the finite order entire solutions $f$ of equation (\ref{e1.8}) satisfies
	\begin{itemize}
		\item [(i)] Every solution $f$ satisfies $\rho(f) = \deg Q$ and is of mean type.
		\item [(ii)] If zero is a Borel exceptional value of $f(z)$, then we have $a_{n-1}=\cdots=a_1=P(z)\equiv 0$. Conversely, if $P(z) \equiv 0$ and there exists an $i_0\in\{1, \ldots, n-1\}$ such that $a_{i_0}=0$, then all of $a_j$'s $(j=1,\ldots, n-1)$ must be zero and we have $\lambda(f) <\rho(f)$; otherwise we have $\lambda(f) =\rho(f)$.
		\item [(iii)] If a solution $f$ belongs to $\Gamma_0'$, then $a_{n-1}=\cdots=a_1=P(z)\equiv 0$. Conversely, let $P(z) \equiv 0$ and there exists an $i_0\in\{1, \ldots, n-1\}$ such that $a_{i_0}=0$, then either $\lambda(f)=\rho(f)-1$ for $c_i=c_j$, $1\leq i,j\leq k$ or $f$ belongs to $\Gamma_0'$.
		\item [(iv)] Let $n \geq 3$. If at least one $a_{i_0}=0$ $(i_0=1,2, \ldots, n-1)$ and $p(z) =z^n+a_{n-1}z^{n-1}+\cdots+a_1z$ such that $ card \{z: p(z) =p'(z) =p''(z) =0\} \geq 1$ or $ card \{z: p(z) =p'(z) =0\} \geq 2$, then $P(z)\equiv 0=a_{n-1}=\cdots=a_1= 0$ and $f\in\Gamma_0'$. Moreover, $ card \{z: p(z) =p'(z) =0\} \geq 2$ is not possible.
		\item [(v)] Let $f$ be given by (\ref{e1.2}), which is a solution of (\ref{e1.8}) for $n=2$ and $a_1\neq 0$. Then the following conclusions hold:
			\begin{itemize}
				\item [(a)] when $m\geq 2$, there exists $i,j\in\{1,2,\ldots,m\}$ such that $\omega_i=2\omega_j$. In this case, $f\in\Gamma_2'$.
				\item [(b)] when $m=1$, then $f$ takes the form $f(z) = H_0(z) + H_1(z)e^{\omega_1z^t}$, i.e., $f\in\Gamma_1'$. In this case,
				 \begin{itemize}
				 	\item [(I)] either $t=1$, $\rho(f)=1$ and $H_0(z)$, $H_1(z)$ are polynomials and $Q(z)$ is a polynomial of degree $1$
				 	\item [(II)] or $H_0(z)=-\frac{a_1}{2}$, $P(z)=-\frac{a_1^2}{4}$, $H_1^2(z)=\frac{b_0a_1}{2}q(z)e^{Q_{t-1}(z)}$ and $L(z,f)=b_0H_0(z)$
				 	\item [(III)] or $H_0(z)=-\frac{a_1}{2}$, $P(z)=-\frac{a_1^2}{4}$, $H_1^2(z)=-q(z)e^{Q_{t-1}(z)}\mathcal{A}_1(z)$ and $L(z,f)=\mathcal{A}_1(z)e^{\omega_1z^t}$, where  	 $\mathcal{A}_1(z)=\sum_{i=0}^{k}b_i\tilde{H}_1(z+c_i)e^{\omega_1(z+c_i)^t-\omega_1z^t}$ such that $\tilde{H}_1(z+c_i)$ are the delay-differential polynomial of $H_1(z)$.
				 \end{itemize}
			\end{itemize}
	\end{itemize}
\end{theo}
\begin{rem}
	Note that {\em Cases (i)-(iv)} and {\em (v)} of {\em Theorem \ref{t1.1}} improve {\em Theorems D} and {\em E}, respectively. Also, since $L(z,f)$ includes $f^{(k)}(z + c)$ and $\Delta_cf(z)$, {\em Theorem \ref{t1.1}} improves {\em Theorems B-C} as follows:\par
	{\em {(I)}} {\em Cases (i)}, {\em (ii)} and {\em (iii)-(iv)} of {\em Theorem \ref{t1.1}} improve {\em Case (a)}, {\em (b)} and {\em (1)} of each of {\em Theorems B-C}, respectively.\par
	{\em {(II)}} {\em Case (v)-(b)} of {\em Theorem \ref{t1.1}} improves, respectively, {\em Case (2)} and {\em Case (3)} of {\em Theorem B} and {\em Theorem C}.
\end{rem}
This following three examples clarify {\em Cases (ii)-(iii)}.
\begin{ex}
	Take $L(z,f)=f''(z+c)$. Then the function $f=e^{2z}$ satisfies the equation $f^2-\frac{1}{4}e^{2z}L(z,f)=0$ such that $e^{2c}=1$. Clearly, $0=\lambda(f)=\rho(f)-1$. This example clarifies {\em Theorem B} as well.
\end{ex}
\begin{ex}
	Let $L(z,f)=\Delta_cf(z)$. Then the function $f=e^{\alpha z}$ satisfies the equation $f^2-\frac{1}{2}e^{\alpha z}L(z,f)=0$ such that $e^{\alpha c}=3$. Clearly, $0=\lambda(f)=\rho(f)-1$. This example also satisfies {\em Theorem C}.
\end{ex}
\begin{ex}
	Let $L(z,f)=f(z+1)+f'(z+1)-f''(z+1)$. Then the function $f=(z+1)e^{z}$ satisfies the equation $f^2-(z+1)e^{z-1}L(z,f)=0$. Note that, here $c_1=c_2=c_3=1$ and $f\in\Gamma_0'$.
\end{ex}
The next example satisfies {\em Case (iv)}.
\begin{ex}
	Let $L(z,f)=f(z+\log 2)+f''(z+\pi i)$. Then the function $f=e^{iz}$ satisfies the equation $f^3+qe^{2i z}L(z,f)=0$, where $q=\frac{1}{e^{-\pi}-2^i}$. Note that, here $p(z)=z^3$ and $ card \{z: p(z) =p'(z) =p''(z) =0\} = 1$. Also, $a_2=a_1=0\equiv P(z)$ and $f\in\Gamma_0'$.
\end{ex}
By the following example, it is clear that the {\em Case (v)-(a)} occurs significantly.
\begin{ex}
	Take $L(z,f)=f'(z+\log 4)-4f(z+\log 3)$ and $m=2$. Then the function $f=e^{2z}-e^{z}+1$ satisfies the equation $f^2-2f+\frac{1}{4}e^{2z}L(z,f)=-1$. Note that here $f\in\Gamma_2'$.
\end{ex}
The following two examples show that the {\em Case (v)-(b)-(I)} actually holds.	
\begin{ex}
	We take $L(z,f)=f(z+c)$. Then the function $f=d+e^{\alpha z}$ satisfies the equation $f^2-df-e^{\alpha z}L(z,f)=0$ such that $e^{\alpha c}=1$.  Here, $P(z)\equiv 0$.\\ Also, the same function satisfies $f^2-3df+e^{\alpha z}L(z,f)=-2d^2$ such that $e^{\alpha c}=-1$. Here, $f\in\Gamma_1'$. Here, $P(z)\not\equiv 0$. This example is true for {\em Theorem E} as well.
\end{ex}
\begin{ex}
	Put $L(z,f)=f(z+\log 2)+f'(z+\pi i)+f''(z+2\pi i)$ and $m=1$. Then the function $f=2+3e^{z}$ satisfies the equation $f^2-3f-\frac{3}{2}e^{z}L(z,f)=-2$. Here, $P(z)\not\equiv 0$.\\ Also, let $L(z,f)=f(z+\log 3)-f'(z+\log 4)+f''(z+\log 2)$ and $m=1$. Then the function $f=3+e^{z}$ satisfies the equation $f^2-3f-e^{z}L(z,f)=0$.  Here, $P(z)\equiv 0$.
\end{ex}
Next example shows that the {\em Case (v)-(b)-(II)} actually occurs.
\begin{ex}
	Let $L(z,f)=3f(z)+f'(z+\log 2)-3f''(z+2\pi i)$ and $m=1$. Then the function $f=-\frac{a_1}{2}+2e^{3z}$ satisfies the equation $f^2+a_1f+\frac{8}{3a_1}e^{6z}L(z,f)=-\frac{a_1^2}{4}$. Note that here $b_0=3$, $H_0=-\frac{a_1}{2}$ and so, $L(z,f)=-\frac{3a_1}{2}=b_0H_0$.
\end{ex}
Next example shows that the {\em Case (v)-(b)-(III)} actually occurs.
\begin{ex}
	Let $L(z,f)=f(z)-f(z+\log 2)+\frac{1}{2}f'(z+\log 2)+\frac{2}{9}f'(z+\log 3)-\frac{1}{9}f''(z+\log 3)$ and $m=1$. Then the function $f=-\frac{a_1}{2}+ze^{2z}$ satisfies the equation $f^2+a_1f-ze^{2z}     L(z,f)=-\frac{a_1^2}{4}$. Note that here, $q(z)=-z$, $Q(z)=2z$, $Q_{t-1}(z)=0$. Also, $\mathcal{A}_1(z)=\sum_{i=0}^{5}\left(b_i\tilde{H}_1(z+c_i)e^{2c_i}\right)=z$. So, $H_1^2(z)=z^2=-q(z)e^{Q_{t-1}(z)}\mathcal{A}_1(z)$ and $L(z,f)=\mathcal{A}_1(z)e^{2z}=\sum_{i=0}^{5}\left(b_i\tilde{H}_1(z+c_i)e^{2c_i}\right)e^{2z}=ze^{2z}$.
\end{ex}
\section{Lemmas}
We give the following well-known results which are important to prove our theorems.
\begin{lem}\label{l1}\cite{Chiang-Feng}
	Let $f$ be a non-constant meromorphic function and $c_1$, $c_2$ be two complex numbers such that $c_1\neq c_2$. Let $f(z)$ be a meromorphic function with finite order $\rho$, then for each $\epsilon > 0$,
	\beas m\left(r,\frac{f(z+c_1)}{f(z+c_2)}\right)=S(r,f).\eeas
\end{lem}
\begin{lem}\cite[Corollary 2.3.4]{Laine_Gruyter}\label{l2}
	Let $f$ be a transcendental meromorphic function and $k\geq 1$ be an integer. Then $m\left(r,\frac{f^{(k)}}{f}\right)= S(r,f)$. 
\end{lem}
Combining {Lemmas \ref{l1}-\ref{l2}} we have the following lemma:
\begin{lem}\label{l3}
	Let $f(z)$ be a meromorphic function of finite order and let $c\in\mathbb{C}$, $k\geq 1$ be an integer. Then $m\left(r,\frac{f^{(k)}(z+c)}{f}\right)= S(r,f)$.
\end{lem}
\begin{proof}
	$$ m\left(r,\frac{f^{(k)}(z+c)}{f}\right)=m\left(r,\frac{f^{(k)}(z+c)}{f(z+c)}.\frac{f(z+c)}{f(z)}\right)=S(r,f).$$
\end{proof}
\begin{lem}\label{l4}\cite{Li-Yang_jmaa_2017} Let $f$ be a non-constant meromorphic function of hyper order less than $1$ and $c\in\mathbb{C}$. Then \beas N(r,1/f(z+c))= N(r,0;f(z))+S(r,f).\eeas
\end{lem}
\begin{lem}\label{l5}\cite{Yang-HX-Kluwer}
	Suppose $f_j (z)$ $( j =1 ,2,...,n+1)$ and $g_k (z)$ $( k =1 ,2,...,n)$ $(n \geq 1)$ are entire functions satisfying the following conditions:
	\begin{itemize}
		\item[(i)] $ \sum_{j=1}^{n} f_j (z)e^{g_j(z)} \equiv f_{n+1} (z)$, \item[(ii)] The order of $f_j (z)$ is less than the order of $e^{g_k(z)}$ for $1 \leq j \leq n +1$ , $1 \leq k \leq n$ and furthermore, the order of $f_j (z)$ is less than the order of $e^{g_h(z)-g_k(z)}$ for $n \geq 2$ and $1\leq j \leq n +1$, $1\leq h<k\leq n$. Then $f_j (z)\equiv 0$, $(j =1 ,2,...,n+1)$.
	\end{itemize} 
\end{lem}
\begin{lem}\label{l6}\cite{Chen_Science Press,Hayman_Oxford}
Let $f$ be a meromorphic function and suppose that
$$R(z)=a_nf(z)^n+\cdots+a_0(z)$$
has small meromorphic coefficients $a_j(z)$, $a_n\neq 0$ in the sense of $T(r, a_j)=S(r, f)$. Moreover, assume that
$$\ol N\left(r,\frac{1}{R}\right)+\ol N(r, f)=S(r, f).$$
Then $$R(z)=a_n\left(f+\frac{a_{n-1}}{na_n}\right).$$
\end{lem}
The following lemma gives the Nevanlinna characteristic and counting functions of an exponential polynomial.
\begin{lem}\label{l7}\cite{Steinmetz}
	Let $f(z)$ be given by (\ref{e1.2}). Then $$T (r, f) = C(co(W_0))\frac{r^t}{2\pi}+ o(r^t).$$ If $H_0(z) \not\equiv 0$, then $$m(r,1/f)= o(r^t),$$ while if $H_0(z) \equiv 0$, then $$N(r,1/f)=C(co(W))\frac{r^t}{2\pi}+ o(r^t).$$
\end{lem}
Next, we proof the following lemmas which are the core parts of our paper.
\begin{lem}\label{l8}
	Let $f$ be given by (\ref{e1.2}) which is a solution of (\ref{e1.8}) for $n=2$ and $\omega_i\neq 2\omega_j$. If the points $0,\omega_1,\omega_2,\ldots,\omega_m$ are collinear, then $m=1$.
\end{lem}
\begin{proof} [\bf\underline{Proof}]
	Assume on the contrary to the assertion that $m\geq 2$. For each $i \in\{1,2,\ldots, m\}$, we may write $\omega_i=\xi_i\omega$, where the constants $\xi_i\in\mathbb{C}\backslash\{0\}$ are distinct, $\xi_0=0$ and $\omega\in\mathbb{C}\backslash \{0\}$. Moreover, we may suppose that $\xi_i>\xi_j$ for $i >j$. Equation (\ref{e1.8}) can be written as \bea\label{e2.1} &&\sum_{i,j=0}^{m} H_i(z)H_j(z)e^{(\xi_i+\xi_j)\omega z^t}+a_1\sum_{l=0}^{m}H_l(z)e^{\xi_l\omega z^t}\nonumber\\&&\qquad+q(z)e^{Q_{t-1}(z)}\left[{\mathcal{A}_0(z)e^{v_tz^t}}+\sum_{h=1}^{m}\mathcal{A}_h(z)e^{(v_t+\xi_h\omega) z^t}\right]=P(z),\eea where $Q_{t-1}(z) =Q(z)-v_tz^t$ with $\deg Q_{t-1}(z)\leq t-1$ and $\mathcal{A}_0(z)=\sum_{i=0}^{k}b_iH^{(r_i)}_0(z+c_i)$, $\mathcal{A}_h(z)=\sum_{i=0}^{k}b_i\tilde{H}_h(z+c_i)e^{\omega_h(z+c_i)^t-\omega_hz^t}$, $h=1,2,\ldots,m$ such that $\tilde{H}_h(z+c_i)$ are the delay-differential polynomial of $H_h(z)$.\\
	Now we consider following two cases to derive contradiction.\\
	
	{\bf{\underline{Case 1.}}} Let $\xi_m>0$. Note that $ \max\{\xi_i+\xi_j:i,j=0,1,\ldots,m\}=2\xi_m$. 
	Since $L(z,f)\not\equiv 0$, then at least one of $\mathcal{A}_h(z)$, $h=0,1,\ldots,m$ is not vanishing.\\
	{\bf{\underline{Case 1.1.}}} Let all $\mathcal{A}_h(z)=0$, $h=1,2,\ldots,m$. Then $\mathcal{A}_0(z)\not\equiv 0$., i.e., $H_0(z)\not\equiv 0$. If $2\xi_m\omega\neq v_t$, applying {\em Lemma \ref{l5}} on (\ref{e2.1}), we obtain $H_m^2(z)\equiv 0$, a contradiction. Next, let $2\xi_m\omega= v_t$. Since, $\omega_i\neq 2\omega_j$, applying {\em Lemma \ref{l5}} on (\ref{e2.1}), we obtain $H_1^2(z)\equiv 0$, a contradiction.\\
	{\bf{\underline{Case 1.2.}}} Let at least one of $\mathcal{A}_h(z)\not=0$, for $h=1,2,\ldots,m$.\\
	{\bf{\underline{Case 1.2.1.}}} Let $\mathcal{A}_0(z)\neq 0$. Then by {\em Lemma \ref{l5}}, from (\ref{e2.1}), there exists one $h_0\in\{0,1,\ldots,m\}$ such that $2\xi_m\omega=v_t+\xi_{h_0}\omega$. Otherwise, we have $H_m^2(z)\equiv 0$, a contradiction.\\
	{\bf{\underline{Case 1.2.1.1.}}} If $h_0=m$, then we have $v_t=\xi_m\omega$. 
	Since $2\xi_i\neq \xi_j$, $j=0,1, \ldots, m$ and $2\xi_1\not\in\{\xi_i+\xi_j:0\leq i,j\leq m,(i,j)\neq(1,1)\}$ and $2\xi_1\not\in\{\xi_m+\xi_i:i=0,1,\ldots,m\}$. By {\em Lemma \ref{l5}}, we obtain $H_1^2(z)\equiv 0$, a contradiction.\\	
	{\bf{\underline{Case 1.2.1.2.}}} If $h_0\in\{0,1,\ldots,m-1\}$, since $0=\xi_0<\xi_1<\xi_2<\cdots<\xi_{m-1}<\xi_m$ and $2\xi_i\neq \xi_j$,  $i,j=0,1, \ldots, m$, then for $m >h_0$, \beas 2\xi_m-\xi_{h_0}+\xi_m>\max\{2\xi_m-\xi_{h_0}+\xi_i:i=0,1,\ldots,m-1\}.\eeas Also, $2\xi_m=\max\{\xi_i+\xi_j:i,j=0,1,\ldots,m\}$. In view of {\em Lemma \ref{l5}}, we obtain $q(z)e^{Q_{t-1}(z)}\mathcal{A}_m\equiv 0$, i.e., $q(z)\equiv 0$, a contradiction.\\
	{\bf{\underline{Case 1.2.2.}}} Let $\mathcal{A}_0(z)= 0$ and $H_0(z)\neq 0$. Then (\ref{e2.1}) becomes \bea\label{e2.2} \sum_{i,j=0}^{m} H_i(z)H_j(z)e^{(\xi_i+\xi_j)\omega z^t}+a_1\sum_{l=0}^{m}H_l(z)e^{\xi_l\omega z^t}+q(z)e^{Q_{t-1}(z)}\sum_{h=1}^{m}\mathcal{A}_h(z)e^{(v_t+\xi_h\omega) z^t}=P(z).\eea Then similar as {\em Case 1.2.1}, by {\em Lemma \ref{l5}}, from (\ref{e2.2}), we have there exists one $h_0\in\{1,2,\ldots,m\}$ such that $2\xi_m\omega=v_t+\xi_{h_0}\omega$ and proceeding similarly as done in {\em Case 1.2.1}, we can get a contradiction.\\	
	{\bf{\underline{Case 1.2.3.}}} Let $\mathcal{A}_0(z)= 0$ and $H_0(z)= 0$. Then (\ref{e2.1}) becomes \bea\label{e2.3} \sum_{i,j=1}^{m} H_i(z)H_j(z)e^{(\xi_i+\xi_j)\omega z^t}+a_1\sum_{l=1}^{m}H_l(z)e^{\xi_l\omega z^t}+q(z)e^{Q_{t-1}(z)}\sum_{h=1}^{m}\mathcal{A}_h(z)e^{(v_t+\xi_h\omega) z^t}=P(z).\eea Next, similar as {\em Case 1.2.1}, by {\em Lemma \ref{l5}}, from (\ref{e2.3}), we have there exists one $h_0\in\{1,2,\ldots,m\}$ such that $2\xi_m\omega=v_t+\xi_{h_0}\omega$ and adopting the same method as done in {\em Case 1.2.1}, we get a contradiction.\\
	
	{\bf{\underline{Case 2.}}} $\xi_m<0$. Note that $\min\{\xi_i+\xi_j:i,j=0,1,\ldots,m\}=2\xi_1.$
	Similar as {\em Case 1}, we divide the following cases.\\
	{\bf{\underline{Case 2.1.}}} Let all $\mathcal{A}_h(z)=0$, $h=1,2,\ldots,m$. Then $\mathcal{A}_0(z)\not\equiv 0$., i.e., $H_0(z)\not\equiv 0$. If $2\xi_1\omega\neq v_t$, applying {\em Lemma \ref{l5}} on (\ref{e2.1}), we obtain $H_1^2(z)\equiv 0$, a contradiction. Next, let $2\xi_1\omega= v_t$. Since, $\omega_i\neq 2\omega_j$, applying {\em Lemma \ref{l5}} on (\ref{e2.1}), we obtain $H_m^2(z)\equiv 0$, a contradiction.\\
	{\bf{\underline{Case 2.2.}}} Let at least one of $\mathcal{A}_h(z)\neq0$ for $h=1,2,\ldots,m$.\\
	{\bf{\underline{Case 2.2.1.}}} Let $\mathcal{A}_0(z)\neq 0$. Then by {\em Lemma \ref{l5}}, there exists one $h_0\in\{0,1,\ldots,m\}$ such that $2\xi_1\omega=v_t+\xi_{h_0}\omega$. Otherwise, we have $H_1^2(z)\equiv 0$, a contradiction.\\
	{\bf{\underline{Case 2.2.1.1.}}} If $h_0=1$, then we have $v_t=\xi_1\omega$.\\ Since, $2\xi_i\neq\xi_j$ and $2\xi_m\not\in\{\xi_i+\xi_j:0\leq i,j\leq m, (i,j)\neq (m,m)\}$ and $2\xi_m\not\in\{\xi_1+\xi_i:i=0,2,\ldots,m\}$. By {\em Lemma \ref{l5}}, we obtain $H_m^2(z)\equiv 0$, a contradiction.\\	
	{\bf{\underline{Case 2.2.1.2.}}} If $h_0\in\{0,2,3,\ldots,m\}$, since $0=\xi_0<\xi_1<\xi_2<\cdots<\xi_{m-1}<\xi_m$  and $2\xi_i\neq \xi_j$,  $i,j=0,1, \ldots, m$, then \beas 2\xi_1-\xi_{h_0}+\xi_1<\min\{2\xi_1-\xi_{h_0}+\xi_i:i=0,2,3,\ldots,m\}.\eeas Also, $\min\{\xi_i+\xi_j:i,j=0,1,\ldots,m\}=2\xi_1$. In view of {\em Lemma \ref{l5}}, we obtain $q(z)e^{Q_{t-1}}\mathcal{A}_1(z)\equiv 0$, i.e., $q(z)\equiv 0$, a contradiction.\\
	{\bf{\underline{Case 2.2.2.}}} Let $\mathcal{A}_0(z)= 0$ and $H_0(z)\neq 0$. Then in this case, we get equation (\ref{e2.2}). Similar as {\em Case 2.2.1}, by {\em Lemma \ref{l5}}, there exists one $h_0\in\{2,3,\ldots,m\}$ such that $2\xi_1\omega=v_t+\xi_{h_0}\omega$ and proceeding similarly as adopted in {\em Case 2.2.1}, we can get a contradiction.\\
	{\bf{\underline{Case 2.2.3.}}} Let $\mathcal{A}_0(z)= 0$ and $H_0(z)= 0$. Then, we have equation (\ref{e2.3}). Similar as {\em Case 2.2.1}, by {\em Lemma \ref{l5}}, there exists one $h_0\in\{2,3,\ldots,m\}$ such that $2\xi_1\omega=v_t+\xi_{h_0}\omega$. Next, adopting the same method as executing in {\em Case 2.2.1}, we get a contradiction.
\end{proof}
\begin{lem}\label{l9}
	If $m\geq 2$ and $\omega_i\neq 2\omega_j$ for any $i\neq j$, then $f$ of the form (\ref{e1.2}) is not a solution of (\ref{e1.8}) for $n=2$.
\end{lem}
\begin{proof} [\bf\underline{Proof}]
	Suppose on the contrary to the assertion that, $m\geq 2$. Substituting $f$ of the form (\ref{e1.2}) into (\ref{e1.8}), we get \bea\label{e2.4}F(z)&=&f^2(z)+a_1f(z)-P(z)\nonumber\\&=& G(z)+\sum_{\scriptsize{ {\begin{array}{clcr}i,j=0\\\omega_i+\omega_j\neq0\end{array}}}}^{m} H_i(z)H_j(z)e^{(\omega_i+\omega_j)z^t}+a_1\sum_{l=1}^{m}H_l(z)e^{\omega_lz^t},\eea where $G(z)=H_0(z)(H_0(z)+a_1)-P(z)$ is either an exponential polynomial of degree $<t$ or a polynomial in $z$.\\Therefore, also
	\bea\label{e2.5}F(z)=-q(z)e^{Q(z)}L(z,f)=-q(z)e^{Q(z)}\sum_{h=0}^{m}\mathcal{A}_h(z)e^{\omega_h z^t},\hspace{4.1cc}\eea such that $\mathcal{A}_h(z)$ is defined as in(\ref{e2.1}).\\
	Now we set \beas X_1&=&\{\ol\omega_1,\ldots,\ol\omega_m,\ol\omega_i+\ol\omega_j:\ol\omega_i+\ol\omega_j\neq 0,i,j=1,\ldots,m\},\\ X_2&=& \{\ol\omega_1,\ldots,\ol\omega_m,2\ol\omega_1,\ldots,2\ol\omega_m\},\\ X_3&=&\{2\ol\omega_1,\ldots,2\ol\omega_m\}.\eeas
	Clearly, by the theory of convexity, we have $\ol\omega_i+\ol\omega_j=\frac{1}{2}\cdot2\ol\omega_i+(1-\frac{1}{2})\cdot2\ol\omega_j$, i.e., $co(X_1)=co(X_2)$. Since, $X_3\subset X_2$, we have $co(X_3)\leq co(X_2)$, respectively.\\
	Next, we consider the following cases to show a contradiction.\\
	{\bf{\underline{Case 1.}}} If all $\mathcal{A}_h(z)=0$ for $h=1,\ldots,m$, then we have $\mathcal{A}_0(z)\neq 0$, which implies $H_0(z)\neq 0$. Then (\ref{e2.5}) becomes $F(z)=-q(z)e^{Q(z)}\mathcal{A}_0(z)$. Then applying {\em Lemma \ref{l7}}, we get \bea\label{e2.6} N\left(r,\frac{1}{F(z)}\right)=N\left(r,\frac{1}{\mathcal{A}_0(z)}\right)=o(r^t).\eea
	{\bf{\underline{Sub-case 1.1.}}} Let $G(z)\equiv 0$. Applying {\em Lemma \ref{l7}} on (\ref{e2.4}), we have \bea\label{e2.7} N\left(r,\frac{1}{F(z)}\right)=C(co(X_1))\frac{r^t}{2\pi}+o(r^t).\eea
	Therefore, (\ref{e2.6}) and (\ref{e2.7}) yields a contradiction.\\
	{\bf{\underline{Sub-case 1.2.}}} Let $G(z)\not\equiv 0$. Applying {\em Lemma \ref{l7}} on (\ref{e2.4}), we have $m\left(r,\frac{1}{F(z)}\right)=o(r^t)$ and then \bea\label{e2.8} m\left(r,\frac{1}{F(z)}\right)+N\left(r,\frac{1}{F(z)}\right)&=&T(r,F(z))+O(1)=2T(r,f(z))+S(r,f)\nonumber\\&=&2\left(C(co(W_0))\frac{r^t}{2\pi}+o(r^t)\right)+S(r,f)\nonumber\\\implies N\left(r,\frac{1}{F(z)}\right)&=&2C(co(W_0))\frac{r^t}{2\pi}+o(r^t).\eea
	Therefore, using (\ref{e2.6}) and (\ref{e2.8}), we get a contradiction.\\
	{\bf{\underline{Case 2.}}} Let there exists some $h_0\in\{1,2,\ldots,m\}$ such that $\mathcal{A}_{h_0}(z)\neq 0$. Now, we denote the following set as \beas V=\{\ol\omega_{h_0}:h_0\in\{1,2,\ldots,m\}\text{ for which }\mathcal{A}_{h_0}(z)\neq 0\}\text{ and }V_0=V\cup\{0\}.\eeas Since, $V\subseteq W$ and $V_0\subseteq W_0$, then $C(co(V))\leq C(co(W))$ and $C(co(V_0))\leq C(co(W_0))$, respectively.\\
	{\bf{\underline{Case 2.1.}}} Let $\mathcal{A}_0(z)\equiv 0$ and $H_0(z)\equiv 0$. Then using {\em Lemma \ref{l7}} on (\ref{e2.5}), we have \bea\label{e2.9} N\left(r,\frac{1}{F(z)}\right)=N\left(r,\frac{1}{L(z,f)}\right)+O(\log r)=C(co(V))\frac{r^t}{2\pi}+o(r^t).\eea
	{\bf{\underline{Case 2.1.1.}}} If $G(z)\not\equiv 0$. Then similar as {\em Sub-case 1.2}, we have equation (\ref{e2.8}). From (\ref{e2.8}) and (\ref{e2.9}), we obtain a contradiction by $C(co(W_0))\geq C(co(V_0))\geq C(co(V))=2C(co(W_0))$.\\ 
	 {\bf{\underline{Case 2.1.2.}}} If $G(z)\equiv 0$, using {\em Lemma \ref{l7}} on (\ref{e2.4}), we have equation (\ref{e2.7}). 
	 Using (\ref{e2.7}) and (\ref{e2.9}), from $C(co(W))\geq C(co(V))=C(co(X_1))=C(co(X_2))\geq C(co(X_3))=2C(co(W))$, we get a contradiction.\\
	{\bf{\underline{Case 2.2.}}} Let $\mathcal{A}_0(z)\equiv 0$ and $H_0(z)\not\equiv 0$. Then proceeding similarly as done in {\em Case 2.1}, we get a contradiction.\\
	{\bf{\underline{Case 2.3.}}} Let $\mathcal{A}_0(z)\neq 0$, which implies $H_0(z)\not\equiv 0 $. Then using {\em Lemma \ref{l7}} on (\ref{e2.5}), we have $m\left(r,\frac{1}{L(z,f)}\right)=o(r^t)$ and then \bea\label{e2.10} N\left(r,\frac{1}{F(z)}\right)&=&N\left(r,\frac{1}{L(z,f)}\right)+O(\log r)\nonumber\\&=&T(r,L(z,f))+o(r^t)=C(co(V_0))\frac{r^t}{2\pi}+o(r^t).\eea
	{\bf{\underline{Case 2.3.1.}}} If $G(z)\not\equiv 0$. Then, from (\ref{e2.8}) and (\ref{e2.10}), we get $C(co(W_0))\geq C(co(V_0))=2C(co(W_0))$, a contradiction.\\
	{\bf{\underline{Case 2.3.2.}}} If $G(z)\equiv 0$. Using (\ref{e2.7}) and (\ref{e2.10}), we get \bea\label{e2.11}C(co(V_0))=C(co(X_1)).\eea
	Now, since $m\geq 2$, by {\em Lemma \ref{l8}}, $co(W_0)$ can not be a line-segment. Therefore, $co(W_0)$ must be a polygon with non-empty interior. If $0$ is not a boundary point of $co(W_0)$, then we have $co(W_0) = co(W)$. Then we have $C(co(W))=C(co(W_0))\geq C(co(V_0))=C(co(X_1))=C(co(X_2))\geq C(co(X_3))=2C(co(W))$, a contradiction. So, $0$ is  a boundary point of $co(W_0)$. We choose the other non-zero corner points of $co(W_0)$ among the points $\ol\omega_1,\ldots,\ol\omega_m$ are $u_1,\ldots,u_t$, $t\leq m$ such that $0\leq \arg(u_i)\leq\arg(u_{i+1})\leq2\pi$ for $1\leq i\leq t-1$. Hence, \bea\label{e2.12} C(co(W_0))=|u_1|+|u_2-u_1|+\cdots+|u_t-u_{t-1}|+|u_t|. \eea Let $X_4=\{u_1,2u_1,2u_2,\ldots,2u_t,u_t\}$. Therefore, the points $2u_1,2u_2,\ldots,2u_t$ are the corner points of $co(X_4)$. However, since $t\leq m$, $co(X_4)$ may have more corner points. Then, using (\ref{e2.12}), we have  \beas C(co(X_3))>C(co(X_4))&>&|2u_1-u_1|+|2u_2-2u_1|+\cdots+|2u_t-2u_{t-1}|+|u_t-2u_t|\nonumber\\&=& |u_1|+2|u_2-u_1|+\cdots+2|u_t-u_{t-1}|+|u_t|\nonumber\\&>&|u_1|+|u_2-u_1|+\cdots+|u_t-u_{t-1}|+|u_t|\nonumber\\&=&C(co(W_0))\eeas
	Therefore, $$C(co(X_1))=C(co(X_2))\geq C(co(X_3))>C(co(W_0))\geq C(co(V_0)),$$ which contradicts (\ref{e2.11}).\\
	Hence, the proof is completed.
\end{proof}
\begin{lem}\label{l10}
	Let $f$ be given by (\ref{e1.2}), which is a solution of (\ref{e1.8}) for $n=2$, then $f$ takes the form, $$f(z) = H_0(z) + H_1(z)e^{\omega_1z^t},$$
	 i.e., $f\in\Gamma_1'$. In this case,
	\begin{itemize}
		\item [(I)] either $t=1$, $\rho(f)=1$ and $H_0(z)$, $H_1(z)$ are polynomials and $Q(z)$ is a polynomial of degree $1$
		\item [(II)] or $H_0(z)=-\frac{a_1}{2}$, $P(z)=-\frac{a_1^2}{4}$, $H_1^2(z)=\frac{b_0a_1}{2}q(z)e^{Q_{t-1}(z)}$ and $L(z,f)=b_0H_0(z)$
		\item [(III)] or $H_0(z)=-\frac{a_1}{2}$, $P(z)=-\frac{a_1^2}{4}$, $H_1^2(z)=-q(z)e^{Q_{t-1}(z)}\mathcal{A}_1(z)$ and $L(z,f)=\mathcal{A}_1(z)e^{\omega_1z^t}$, where  	 $\mathcal{A}_1(z)=\sum_{i=0}^{k}b_i\tilde{H}_1(z+c_i)e^{\omega_1(z+c_i)^t-\omega_1z^t}$ such that $\tilde{H}_1(z+c_i)$ are the delay-differential polynomial of $H_1(z)$.
	\end{itemize}
\end{lem}
\begin{proof} [\bf\underline{Proof}]
	For $n=2$, (\ref{e1.8}) becomes \bea\label{e2.13} f^{2}(z)+a_{1}f(z)+q(z)e^{Q(z)}L(z,f)=P(z).\eea By {\em Lemma \ref{l9}}, we have $m=1$, i.e., (\ref{e1.2}) becomes \bea\label{e2.14}f(z) = H_0(z) + H_1(z)e^{\omega_1z^t},\eea where $H_0(z)$, $H_1(z)(\not\equiv 0)$ are either exponential polynomials of order $<t$ or ordinary polynomials in $z$. Substituting (\ref{e2.14}) in (\ref{e2.13}), we have \bea\label{e2.15} &&H_1(z)(2H_0(z)+a_1)e^{\omega_1z^t} + H_1^2(z)e^{2\omega_1z^t} +q(z)e^{Q_{t-1}(z)}{\mathcal{A}_0(z)e^{v_tz^t}}\nonumber\\&&\quad+q(z)e^{Q_{t-1}(z)}\mathcal{A}_1(z)e^{(v_t+\omega_1) z^t}=P(z)-H_0(z)(H_0(z)+a_1),\eea where $Q_{t-1}(z) =Q(z)-v_tz^t$ with $\deg Q_{t-1}(z)\leq q-1$ and $\mathcal{A}_0(z)=\sum_{i=0}^{k}b_iH^{(r_i)}_0(z+c_i)$, $\mathcal{A}_1(z)=\sum_{i=0}^{k}b_i\tilde{H}_1(z+c_i)e^{\omega_1(z+c_i)^t-\omega_1z^t}$ such that $\tilde{H}_h(z+c_i)$ are the delay-differential polynomial of $H_h(z)$ for $h=1,2$.
	Since, $L(z,f)\not\equiv 0$, then at least one of $\mathcal{A}_0(z)$ and $\mathcal{A}_1(z)$ is non-vanishing. Next, we divide the following cases to prove our result.\\
	{\bf{\underline{Case 1.}}} Let $\mathcal{A}_1(z)\equiv 0$. Then $\mathcal{A}_0(z)\not\equiv 0$, which implies $H_0(z)\not\equiv 0$.\\
	If $v_t\neq \omega_1,2\omega_1$ or $v_t=\omega_1$, applying {\em Lemma \ref{l5}} on (\ref{e2.15}), we have $H_1^2(z)\equiv 0$, a contradiction.\\
	If $v_t=2\omega_1$, then, applying {\em Lemma \ref{l5}} on (\ref{e2.15}), we have $$H_1(2H_0(z)+a_1)= 0,$$ $$ H_1^2(z)+q(z)e^{Q_{t-1}(z)}{\mathcal{A}_0(z)}=0,$$ $$P(z)-H_0(z)(H_0(z)+a_1)=0.$$ Since, $H_1(z)\not\equiv 0$. Therefore, solving these three equations, we have $H_0(z)=-\frac{a_1}{2}$, $P(z)=-\frac{a_1^2}{4}$ and $H_1^2(z)=\frac{b_0a_1}{2}q(z)e^{Q_{t-1}}$. Therefore, in this case $L(z,f)=b_0H_0(z)$.\\
	{\bf{\underline{Case 2.}}} Let $\mathcal{A}_0(z)\equiv 0$. Then $\mathcal{A}_1(z)\not\equiv 0$.\\
	{\bf{\underline{Sub-case 2.1.}}} Let $H_0(z)\equiv 0$. If $v_t=\pm\omega_1$, using {\em Lemma \ref{l5}} on (\ref{e2.15}), we have $H_1\equiv 0$, a contradiction.\\
	{\bf{\underline{Sub-case 2.2.}}} Let $H_0(z)\not\equiv 0$. If $v_t=-\omega_1$, in view of {\em Lemma \ref{l5}}, from on (\ref{e2.15}), we have $H_1\equiv 0$, a contradiction. If $v_t=\omega_1$, similar as {\em Case 1}, we have $H_0(z)=-\frac{a_1}{2}$, $P(z)=-\frac{a_1^2}{4}$ and $H_1^2(z)=-q(z)e^{Q_{t-1}}\mathcal{A}_1$. Also, in this case $L(z,f)=\mathcal{A}_1e^{\omega_1z^t}$.\\
	{\bf{\underline{Case 3.}}} Let $\mathcal{A}_0(z)\not\equiv 0$ and $\mathcal{A}_1(z)\not\equiv 0$, which implies $H_0(z)\not\equiv 0$.\\
	If $v_t=-\omega_1$ or $v_t\neq\pm\omega_1$, using {\em Lemma \ref{l5}} on (\ref{e2.15}), we have $H_1\equiv 0$, a contradiction. If $v_t=2\omega_1$, by {\em Lemma \ref{l5}}, from (\ref{e2.15}), we have $\mathcal{A}_1\equiv 0$, a contradiction.\\
	If $v_t=\omega_1$, applying {\em Lemma \ref{l5}} on (\ref{e2.15}), we have 
	\bea\label{e2.16}H_1(z)(2H_0(z)+a_1)+q(z)e^{Q_{t-1}(z)}\sum_{i=0}^{k}b_iH^{(r_i)}_0(z+c_i)=0,\eea \bea\label{e2.17}H_1^2(z)+q(z)e^{Q_{t-1}(z)}\sum_{i=0}^{k}b_i\tilde{H}_1(z+c_i)e^{\omega_1(z+c_i)^t-\omega_1z^t}=0,\eea \bea\label{e2.18}P(z)-H_0(z)(H_0(z)+a_1)=0.\eea
	Now, we show that $H_0(z)$ is a polynomial. If possible let, $H_0(z)$ is transcendental. Then from (\ref{e2.18}), we have $$2T(r,H_0(z))+S(r,H_0(z))=T(r,P)=O(\log r),$$ a contradiction.\\ Next, from (\ref{e2.16}), we have \bea\label{e2.19} H_1(z)=\beta(z)e^{Q_{t-1}(z)},\eea where $\beta(z)=-\frac{q(z)\sum_{i=0}^{k}b_iH^{(r_i)}_0(z+c_i)}{2H_0(z)+a_1}$. Since, $f$ is entire, then $\beta(z)$ is a polynomial. Substituting (\ref{e2.19}) in (\ref{e2.17}), we have
	\bea\label{e2.20} \beta^2(z)+q(z)\sum_{i=0}^{k}b_i\tilde{\beta}(z+c_i)e^{Q_{t-1}(z+c_i)-Q_{t-1}(z)+\omega_1(z+c_i)^t-\omega_1z^t}=0, \eea where $\tilde{\beta}(z+c_i)$ is a delay-differential polynomial in $H_0(z)$ and $Q_{t-1}(z)$.\\
	Note that $\deg(\omega_1(z+c_i)^t-\omega_1z^t)=t-1$ and $\deg(Q_{t-1}(z+c_i)-Q_{t-1}(z))\leq t-2$. If $t\geq 2$, applying {\em Lemma \ref{l5}} on (\ref{e2.20}), we have $q(z)=0$, a contradiction. Therefore, $t=1$, i.e., $H_1(z)$ is a polynomial. Hence, $f(z)$ reduces to the form $$f(z) = H_0(z) + H_1(z)e^{\omega_1z},$$ where $H_0(z)$ and $H_1(z)$ are polynomials. So, $f\in\Gamma_1'$.
\end{proof}
\section{Proofs of Theorems}
\begin{proof} [\bf\underline{Proof of Theorem \ref{t1.1} (i)}]
	Suppose that $f$ be a finite order non-vanishing entire solution of (\ref{e1.8}). Using {\em Lemma \ref{l5}}, we have $f$ is transcendental. Otherwise, we will get $L(z,f)\equiv 0$, which yields a contradiction. In view of {\em Lemma \ref{l3}}, from (\ref{e1.8}), we get
	\bea\label{e3.1} n~T(r, f)+S(r, f)&=&m\left(r, f^{n}(z)+\sum_{i=1}^{n-1}a_{i}f^{i}(z)\right)\nonumber\\&=&m(r,P(z))-q(z)e^{Q(z)}L(z,f))\nonumber\\&=&m(r,e^{Q(z)})+m(r,L(z,f))+O(1)\nonumber\\&=&m(r,e^{Q(z)})+m\left(r,\frac{L(z,f)}{f(z)}\right)+m(r,f(z))+O(1)\nonumber\\&=&T(r,e^{Q(z)})+T(r,f(z))+S(r,f)\nonumber\\\implies (n-1)~T(r, f)&\leq& T(r,e^{Q(z)})+S(r,f).\eea
	Therefore, for $n\geq 2$, $\rho(f)\leq \deg{Q(z)}$. If $\rho(f)< \deg{Q(z)}$, then comparing order of growth of (\ref{e1.8}), we get a contradiction. So, $\rho(f)= \deg{Q(z)}$. Now, from the definition of type, we have $$\tau(f)=\ol\lim_{r\rightarrow\infty}\frac{T(r,f)}{r^{\rho(f)}}=\ol\lim_{r\rightarrow\infty}\frac{T(r,f)}{r^{\deg{Q(z)}}}\in(0,\infty),$$ i.e., $f$ is of mean type.
\end{proof}
\begin{proof} [\bf\underline{Proof of Theorem \ref{t1.1} (ii)}]
	First, we prove that if zero is a Borel exceptional value of $f(z)$, then we have $a_{n-1}=\cdots=a_1=0\equiv P(z)$. Adopting the similar process as done in the proof of Theorem 1.2(b) in \cite{Li-Yang_jmaa_2017}, using {\em Lemmas \ref{l3}, \ref{l4}, \ref{l6}} and replacing $f(z+c)$ by $L(z,f)$, we can prove our result. In this regards, only the equation (10) of \cite{Li-Yang_jmaa_2017} is replaced by the following lines \beas N\left(r,\frac{1}{G(z)}\right)&=&N\left(r,\frac{1}{q(z)e^{Q(z)}L(z,f)}\right)\nonumber\\&\leq& N\left(r,\frac{1}{q(z)}\right)+N\left(r,\frac{1}{L(z,f)}\right)+S(r,f)\nonumber\\&\leq& (k+1)~N\left(r,\frac{1}{f(z)}\right)+S(r,f)=S(r,f).\eeas
	
	Next, we prove the converse part of {\em Theorem \ref{t1.1} (ii)}. For this, using {\em Lemmas \ref{l3}, \ref{l6}} and replacing $f(z+c)$ by $L(z,f)$, we proceed similar up to equation (17) in the proof of Theorem 1.2(c) in \cite{Li-Yang_jmaa_2017}. Here, the equations (15), (16) and (17) of \cite{Li-Yang_jmaa_2017}, respectively, will be \bea\label{e3.2}T\left(r,\frac{L(z,f)}{f(z)}\right)=S(r,f),\eea 
	\bea\label{e3.3}\left(f(z)+\frac{a_{n-1}}{n-1}\right)^{n-1}=f^{n-1}(z)+\sum_{i=1}^{n-1}a_{i}f^{i-1}(z)=-q(z)e^{Q(z)}\frac{L(z,f)}{f(z)}\eea and \bea\label{e3.4} a_{i+1}=\frac{(n-1)!}{i!(n-1-i)!}\left(\frac{a_{n-1}}{n-1}\right)^{n-1-i},\;\;i=0,1,...,n-2.\eea

	{\bf Case 1.} If there exists $i_0\in\{1,\ldots, n-1\}$ such that $a_{i_0}=0$, then from (\ref{e3.4}), we have all $a_{i}$ must be equal to zero for $i=1,2,\ldots, n-1$. Therefore, (\ref{e3.3}) becomes \beas f(z)^{n-1}=-q(z)e^{Q(z)}\frac{L(z,f)}{f(z)}.\eeas By using (\ref{e3.2}), for each $\epsilon>0$, we have \beas (n-1)N\left(r,\frac{1}{f(z)}\right)&=&N\left(r,\frac{1}{q(z)\frac{L(z,f)}{f(z)}}\right)\\&\leq& N\left(r,\frac{1}{q(z)}\right)+ N\left(r,\frac{1}{\frac{L(z,f)}{f(z)}}\right)+S(r,f)=S(r,f).\eeas Therefore, $\lambda(f)<\rho(f)$.
	
	{\bf Case 2.} If there exists no $i_0\in\{1,\ldots, n-1\}$ such that $a_{i_0}=0$, then from (\ref{e3.2}) and (\ref{e3.3}), we have \beas \ol N\left(r,\frac{1}{f(z)+\frac{a_{n-1}}{n-1}}\right)\leq \ol N\left(r, \frac{1}{q(z)} \right)+\ol N\left(r, \frac{1}{\frac{L(z,f)}{f(z)}} \right)+S(r,f)=S(r,f).\eeas Using the second main theorem, we have \beas T(r,f)&\leq& \ol N\left(r,\frac{1}{f(z)+\frac{a_{n-1}}{n-1}}\right)+\ol N\left(r,\frac{1}{f(z)}\right)+\ol N\left(r,f(z)\right)\\&=&\ol N\left(r,\frac{1}{f(z)}\right)+S(r,f)\eeas Therefore, $\rho(f)\leq\lambda(f)$ but we know that $\lambda(f)\leq\rho(f)$. Therefore, $\lambda(f)=\rho(f)$.
\end{proof}
\begin{proof} [\bf\underline{Proof of Theorem \ref{t1.1} (iii)}]
	Suppose that $f$ be a non-vanishing finite order entire solution of (\ref{e1.8}). Similar as {\em Theorem \ref{t1.1} (i)}, $f$ is transcendental.\par First suppose that $f$ belongs to $\Gamma_0'$, which means that $0$ is a Borel exceptional value of $f$. Thus, from {\em Theorem \ref{t1.1} (ii)}, we have $a_{n-1}=\cdots=a_1=0\equiv P(z)$.\par Next, we suppose that $P(z) \equiv 0$ and there exists an $i_0\in\{1, \ldots, n-1\}$ such that $a_{i_0}=0$, then from the converse part of {\em Theorem \ref{t1.1} (ii)}, we have all of the $a_i (i=1,\ldots, n-1)$ must be zero as well and $\lambda(f) <\rho(f)$. From Hadamard factorization theorem, we can see that \bea\label{e3.5}f(z)=h(z)e^{\alpha(z)},\eea where $\alpha(z)$ is a polynomial and $h(z)$ is the canonical product of zeros of $f$ with $\deg{\alpha(z)}=\rho(f)=\deg{Q(z)}=t$ and $\rho(h)=\lambda(h)=\lambda(f)<\rho(f)$.\\ Substituting (\ref{e3.5}) in (\ref{e1.8}) with all $a_i=0$, we have \bea\label{e3.6} h^n(z)e^{n \alpha(z)}+q(z)e^{Q(z)+\alpha(z)}\left(\sum_{i=0}^{k}L_i(z,h)e^{\Delta_{c_i}\alpha(z)}\right)=0,\eea where \beas L_i(z,h)&=&b_i\left[h(z+c_i)M_{k_i}(\alpha'(z+c_i), \alpha''(z+c_i),\ldots,\alpha^{(k_i)}(z+c_i))\right.\\&&\qquad\left.+h'(z+c_i)M_{k_i-1}(\alpha'(z + c_i), \alpha''(z+c_i),\ldots,\alpha^{(k_i-1)}(z+c_i))\right.\\&&\qquad\left.+ \cdots +h^{(k_i-1)}(z+c_i)M_1(\alpha'(z+c_i))+h^{(k_i)}(z+c_i)\right].\eeas
	Clearly, $\rho(L_i(z,h))< t$. Rewriting (\ref{e3.6}), we have \bea\label{e3.7} h^n(z)e^{n \alpha_{t-1}(z)}e^{n u_tz^t}+q(z)e^{\alpha_{t-1}(z)+Q_{t-1}(z)}\left(\sum_{i=0}^{k}L_i(z,h)e^{\Delta_{c_i}\alpha(z)}\right)e^{(u_t+v_t)z^t}=0\eea such that $\alpha(z)=u_tz^t+\alpha_{t-1}(z)$ and $Q(z)=v_tz^t+Q_{t-1}(z)$, where $u_t$, $v_t$ are non-zero constants and $\alpha_{t-1}(z)$, $Q_{t-1}(z)$ are of degree $\leq t-1$.\\ In view of {\em Lemma \ref{l5}}, we can easily say that (\ref{e3.7}) is possible only when $(n-1)u_t=v_t$. Therefore, (\ref{e3.7}) becomes \bea\label{e3.8} h^n(z)+q(z)e^{(1-n)\alpha_{t-1}(z)+Q_{t-1}(z)}\left(\sum_{i=0}^{k}L_i(z,h)e^{\Delta_{c_i}\alpha(z)}\right)=0.\eea Here, the following cases arise.\\
	{\bf\underline{Case 1:}} Let $\rho(h)<t-1>0$. If $\deg\{(1-n)\alpha_{t-1}(z)+Q_{t-1}(z)\}= t-1$, applying {\em Lemma \ref{l5}}, we have $q(z)=0$, a contradiction. If $\deg\{(1-n)\alpha_{t-1}(z)+Q_{t-1}(z)\}< t-1$, from $\deg \{\Delta_{c_i}\alpha(z)\}=t-1$, by using {\em Lemma \ref{l5}}, again we have $q(z)=0$, a contradiction.\\
	{\bf\underline{Case 2:}} Let $\rho(h)\geq t-1>\rho(h)-1$, $t-1>0$. By logarithmic derivative lemma \cite[Corollary 2.5]{Chiang-Feng}, for each $\epsilon>0$, we have \bea\label{e3.9} m\left(r,\frac{L_i(z,h)}{h(z)}\right)=O(r^{\rho(h)-1+\epsilon})+O(\log r).\eea
	Since, $h$ is entire, using (\ref{e3.8}) and (\ref{e3.9}), we have \bea\label{e3.10} T\left(r,\sum_{i=0}^{k}\frac{L_i(z,h)}{h}e^{\Delta_{c_i}\alpha(z)}\right)&=&m\left(r,\sum_{i=0}^{k}\frac{L_i(z,h)}{h}e^{\Delta_{c_i}\alpha(z)}\right)+N\left(r,\sum_{i=0}^{k}\frac{L_i(z,h)}{h}e^{\Delta_{c_i}\alpha(z)}\right)\nonumber\\&\leq& \sum\limits_{i=0}^{k}T\left(r,e^{\Delta_{c_i}\alpha(z)}\right)+O(r^{\rho(h)-1+\epsilon})+O(\log r).\eea
	Therefore, using (\ref{e3.8}) and (\ref{e3.10}), we obtain for each $\epsilon>0$ \bea\label{e3.11} (n-1)N\left(r,\frac{1}{h(z)}\right)&\leq& N\left(r,\frac{1}{\sum_{i=0}^{k}\frac{L_i(z,h)}{h}e^{\Delta_{c_i}\alpha(z)}}\right)+O(\log r)\nonumber\\&\leq& T\left(r,\sum_{i=0}^{k}\frac{L_i(z,h)}{h}e^{\Delta_{c_i}\alpha(z)}\right)+O(\log r)\nonumber\\&\leq& \sum\limits_{i=0}^{k}T\left(r,e^{\Delta_{c_i}\alpha(z)}\right)+O(r^{\rho(h)-1+\epsilon})+O(\log r).\eea
	Now, if $c_i=c_j$ for all $1\leq i,j\leq k$, say, $c$, then (\ref{e3.11}) becomes \beas (n-1)N\left(r,\frac{1}{h(z)}\right)\leq T\left(r,e^{\Delta_{c}\alpha(z)}\right)+O(r^{\rho(h)-1+\epsilon})+O(\log r).\eeas
	Thus, from the above equation, we have $\lambda(h)\leq t-1$. But in this case, $\lambda(h)=\rho(h)\geq t-1$. Therefore, $\lambda(f)=\lambda(h)=t-1=\rho(f)-1$.\par If $t-1=0$, we get $\rho(h)=\lambda(h)<\rho(f)=t=1$. If $h(z)$ is transcendental, then $h(z)$ has infinitely many zeros. Now noting that $\Delta_{c_i}\alpha(z)$ is of degree $t-1$, from (\ref{e3.11}), we get $N\left(r,\frac{1}{h(z)}\right)=O(\log r)$, a contradiction. Therefore, $h(z)$ will be a polynomial. So, $f$ belongs to $\Gamma_0'$.\\
	{\bf\underline{Case 3:}} Let $\rho(h)\geq t$. Then from (\ref{e3.11}), we obtain $\lambda(h)<\rho(h)$, a contradiction. So, $h(z)$ will be a polynomial. Therefore, $f$ belongs to $\Gamma_0'$.
\end{proof}
\begin{proof} [\bf\underline{Proof of Theorem \ref{t1.1} (iv)}]
	Suppose that $f$ is a non-vanishing finite order entire solution of (\ref{e1.8}). Similarly, $f$ is transcendental. If possible let, $P(z)\not\equiv 0$. By the assumption, $ card \{z: p(z) =p'(z) =p''(z) =0\} \geq 1$ or $ card \{z: p(z) =p'(z) =0\} \geq 2$, where $p(z) =z^n+a_{n-1}z^{n-1}+\cdots+a_1z$, we mean that $p(z)$ has at least one zero with multiplicity at least three or at least $2$ zeros with multiplicities at least two. In view of {\em Lemma \ref{l3}} and the second main theorem, we have \beas n~T(r,f)&=&T(r,f^n(z)+a_{n-1}f^{n-1}(z)+\cdots+a_1f(z))\\&\leq&\ol N\left(r,\frac{1}{f^n+a_{n-1}f^{n-1}+\cdots+a_1f-P(z)}\right)\\&&+\ol N\left(r,\frac{1}{f^n+a_{n-1}f^{n-1}+\cdots+a_1f}\right)+\ol N(r,f^n+a_{n-1}f^{n-1}+\cdots+a_1f)\\&\leq&\ol N\left(r,\frac{1}{q(z)L(z,f)}\right)+(n-2)~T(r,f)+S(r,f)
	\\&\leq& T\left(r,q(z)L(z,f)\right)+(n-2)~T(r,f)+S(r,f)\\&\leq& N\left(r,q(z)L(z,f)\right)+m\left(r,q(z)L(z,f)\right)+(n-2)~T(r,f)+S(r,f)\\&\leq& m\left(r,q(z)\frac{L(z,f)}{f(z)}\right)+m(r,f(z))+(n-2)~T(r,f)+S(r,f)\\&\leq& T(r,f)+(n-2)~T(r,f)+S(r,f)\\&\leq&(n-1)~T(r,f)+S(r,f), \eeas a contradiction. Therefore, $P(z)\equiv 0$.\par Since, at least one $a_{i_0}=0$ $(i_0=1,2, \ldots, n-1)$, using {\em Theorem \ref{t1.1} (iii)}, we have $f\in\Gamma_0'$.\\ Note that, since, $P(z)\equiv 0$ and at least one $a_{i_0}=0$ $(i_0=1,2, \ldots, n-1)$, from {\em Theorem \ref{t1.1} (ii)}, we have all of $a_j$'s $(j=1,\ldots,n-1)$ must be zero. Therefore $p(z)$ is of the form $z^n$, which implies $ card \{z: p(z) =p'(z) =0\} \geq 2$ is not possible.
\end{proof}
\begin{proof} [\bf\underline{Proof of Theorem \ref{t1.1} (v)}]
	From {\em Theorem \ref{t1.1} (i)}, we have $\rho(f)=\deg(Q(z))$. Using {\em Lemmas \ref{l8} - \ref{l10}}, we can prove the result.	
\end{proof}

\end{document}